\documentclass[12pt]{amsart}
\usepackage{amssymb,latexsym,amsthm}
\usepackage{amsmath,amsfonts,amssymb}
\usepackage{graphicx}
\usepackage{color}
\usepackage{mathrsfs}

\newcommand\R{\mathbb R}
\newcommand\C{\mathbb C}

\newcommand{\x}{\mbox{\boldmath $x$}}

\newcommand{\de}{\Delta}
\newcommand{\la}{\lambda}
\newcommand{\1}{\mbox{\boldmath $1$}}

\newcommand{\Lip}{\operatorname{Lip}}
\newcommand{\G}{\operatorname{GWC}}

\newcommand{\Iso}{\operatorname{Iso}}

\newcommand{\id}{\operatorname{Id}}

\newcommand{\Sd}{\operatorname{S}^{\infty}(\mathbb{D})}

\newcommand{\isonly}{\operatorname{Iso}}
\newcommand{\isonlyc}{\isonly_{\mathbb C}}
\newcommand{\Mmc}{M_{\mathbb C}(B_1,B_2)}

\newcommand{\mc}{\mathbb C}
\newcommand{\mr}{\mathbb R}

\newtheorem{theorem}{Theorem}
\newtheorem{lemma}[theorem]{Lemma}
\newtheorem{cor}[theorem]{Corollary}
\newtheorem{prop}[theorem]{Proposition}

\theoremstyle{definition}
\newtheorem{definition}[theorem]{Definition}

\theoremstyle{remark}
\newtheorem{remark}[theorem]{Remark}

\begin{document}
\author{
Shiho Oi
}
\address{
Niigata Prefectural Hakkai High School, Minamiuonuma 
949-6681 Japan.
}
\email{shiho.oi.pmfn20@gmail.com
}

\title[]
{A generalization of the Kowalski -S\l odkowski theorem and its application to 2-local maps on function spaces}




\begin{abstract}
In this paper, we extend a spherical variant of the Kowalski-S\l odkowski theorem due to Li, Peralta, Wang and Wang \cite{lpww}. As a corollary, we prove that every 2-local map in the set of all surjective isometries (without assuming linearity) on a certain function space is in fact a surjective isometry. This gives an affirmative answer to the problem on 2-local isometries posed by Moln\'ar.
\end{abstract}
\maketitle
\section{Introduction}
One of the basic problem in the operator theory is to find sufficient sets of conditions for linearity and multiplicativity of maps between Banach algebras. As a generalization of the Gleason-Kahane-\.Zelazko theorem \cite{gleason,kz1,zelazko}, Kowalski and S\l odkowski \cite{ks} proved the linearity and the multiplicativity of a functional $\de$  on a Banach algebra $A$ under the spectral condition; 
$\de(a)-\de(b) \in \sigma(a-b)$ for $a,b\in A$. Recently Li, Peralta, Wang and Wang proved interesting spherical variants of the Gleason-Kahane-\.Zelazko theorem, and the Kowalski-S\l odkowski \cite{lpww}. They proved that a 1-homogeneous functional on a unital Banach algebra which satisfy a mild spectral condition is linear. Applying it they proved the complex-linearity of certain 2-local maps.

Motivated by the Kowalski-S\l odkowski theorem, the concept of a 2-local map was introduced by \v Semrl \cite{smrl}, who proved the first results on 2-local automorphisms and derivations on algebras of operators. Moln\'ar \cite{mol2l} began to study 2-local complex-linear isometries. 
Given a Banach space $\mathfrak{M}_j$ for $j=1,2$, an isometry from $\mathfrak{M}_1$ into $\mathfrak{M}_2$ is a distance preserving map. The set of all surjective complex-linear isometries from $\mathfrak{M}_1$ onto $\mathfrak{M}_2$ is denoted by $\operatorname{Iso}_{\mc}(\mathfrak{M}_1,\mathfrak{M}_2)$. 
The set of all maps from ${\mathfrak M}_1$ into ${\mathfrak M}_2$ is denoted by $M({\mathfrak M}_1, {\mathfrak M}_2)$. We say that a map $T\in M({\mathfrak M}_1, {\mathfrak M}_2)$ is 2-local complex-linear isometry if for every $x,y\in {\mathfrak M}_1$ there is a $T_{x,y}\in \isonly_{\mc}(\mathfrak{M}_1,\mathfrak{M}_2)$ 
such that $T(x)=T_{x,y}(x)$ and $T(y)=T_{x,y}(y)$. 
Moln\'ar \cite{mol2l} proved that a 2-local complex-linear isometry on a certain $C^*$-algebra is a surjective complex-linear isometry. Initiated by his result, there are a lot of studies on 2-local complex-linear isometries on operator algebras and function spaces assuring that a 2-local complex-linear isometry is in fact a surjective complex-linear isometry \cite{ahhf,bjm,gy,hmto,jlp,jvvv,lpww,molbook,mol2l}.

Analogous  to the problem on a 2-local complex-linear isometry, Moln\'ar raised a problem on a 2-local isometry. 
The set of all surjective isometries  (not necessarily linear) from ${\mathfrak{M}}_1$ onto ${\mathfrak{M}}_2$ is denoted by $\operatorname{Iso}(\mathfrak{M}_1,\mathfrak{M}_2)$. We say that $T\in M({\mathfrak M}_1, {\mathfrak M}_2)$ is a 2-local isometry or $T$ is 2-local in $\operatorname{Iso}(\mathfrak{M}_1,\mathfrak{M}_2)$ if for every $x,y\in {\mathfrak M}_1$ there is a $T_{x,y}\in \operatorname{Iso}(\mathfrak{M}_1,\mathfrak{M}_2)$
such that 
\[
\text{$T(x)=T_{x,y}(x)$ and $T(y)=T_{x,y}(y)$}.
\]
A 2-local isometry is necessarily an isometry by the assumption, but it needs not to be surjective in general.  Once surjectivity of a 2-local map is obtained, it is a surjective  isometry.  
Moln\'ar \cite{mol2loc2} proved that every 2-local isometry in $\isonly(B(H),B(H))$  is an element in $\isonly(B(H),B(H))$, where $B(H)$ is the Banach algebra of all bounded linear operators on a separable Hilbert space $H$.  One may suppose that there is not so big difference between the problem on 2-local complex-linear isometries and one on 2-local isometries, but the fact is that the later is highly nontrivial (cf. \cite{mol2loc2,haoi}).  Moln\'ar posed a question whether a 2-local map in $\isonly(C(X),C(X))$ is an element in $\isonly(C(X),C(X))$ or not for a  first countable compact Hausdorff space $X$ \cite{mo2loc}. (It is known that it is not the case without assuming the first countability.)  Even for $X=[0,1]$ the answer to the problem has not been known. Inspired by his problem,  Hatori and the author proved that a 2-local map in $\isonly(B,B)$ is an element in $\isonly(B,B)$, where $B$ is the Banach space of all continuously differentiable functions or the Banach space of Lipschitz functions on the closed unit interval equipped with a certain norm \cite{haoi}.

The aim of this paper is to show a generalization of a spherical variant of  theorem of Kowalski and S\l odkowski exhibited in \cite{lpww}. Applying it, we prove that 2-local isometries on several function spaces are surjective isometries. In particular, we give an affirmative answer to the problem of Moln\'ar (Corollary \ref{C(X)}). Note that Mori \cite{hamo} also got an affirmative answer for the problem in a different way.

In this paper, we denote the unit circle on the complex plane by $\mathbb{T}=\{z \in \mathbb{C}; |z|=1\}$. For the simplicity of the notation we denote $[f]^1=f$ and $[f]^{-1}=\overline{f}$, the complex-conjugate of $f$ for any complex-valued function $f$. For any unital Banach algebra, $\1$ stands for the unity of itself. The identity map is denoted by $\id$.

\section{generalization of the  Kowalski-S\l odkowski theorem}
Li, Peralta, Wang and Wang \cite{lpww} proved a spherical variant of the Kowalski-S\l odkowski theorem; a 1-homogeneous functional which satisfies certain spectral condition is complex-linear. We shall prove a generalization of spherical variant of the Kowalski-S\l odkowski theorem 
without hypothesis that the 1-homogeneity.
This hypothesis is the same as one of the original Kowalski-S\l odkowski theorem.
\begin{theorem}\label{gks}
Let $A$ be a unital Banach algebra. Suppose that a map $\de: A \to \C$ satisfies the conditions
\def\theenumi{\alph{enumi}}
\begin{enumerate}
\item $\de(0)=0$,
\item $\de(x)-\de(y) \in \mathbb{T}\sigma(x-y), \quad x,y \in A$.
\end{enumerate}
Then $\de$ is a complex-linear or conjugate linear and $\overline{\de(\1)}\de$ is multiplicative.
\end{theorem}

Fix $a \in A$, we define a map $f: \C \to \C$ by $f(\la)=\de(a+\la\cdot \1)-\de(a)$. For any $\la_1, \la_2 \in \C$, we get
\[
\de(a+\la_1\cdot \1)-\de(a+\la_2\cdot \1) \in \mathbb{T} \sigma((\la_1-\la_2)\cdot \1)=(\la_1-\la_2)\mathbb{T}\sigma(\1)=(\la_1-\la_2)\mathbb{T},
\]
by the assumption (b). Thus we have 
\begin{equation*}
\begin{split}
|f(\la_1)-f(\la_2)|&=|\de(a+\la_1\cdot \1)-\de(a)-\{\de(a+\la_2\cdot \1)-\de(a)\} |\\
&=|\de(a+\la_1\cdot \1)-\de(a+\la_2\cdot \1) |\\
&=|\la_1-\la_2|.
\end{split}
\end{equation*}
This implies that the map $f$ is an isometry on $\C$. The form of an isometry on $\C$ is well known. Without assuming surjectivity on the isometry there exist $\alpha,\beta \in \C$ with $|\alpha|=1$ such that $f(\lambda)=\beta+\lambda\alpha$ ($\lambda\in \C$) or $f(\lambda)=\beta+\bar{\lambda}\alpha$ ($\lambda \in \C$). Since 
\[
f(0)=\de(a+0 \cdot \1)-\de(a)=\de(a)-\de(a)=0,
\]
 we have
\[
f(\la)=\la \alpha, \quad \la \in \C,
\]
or
\[
f(\la)=\overline{\la} \alpha, \quad \la \in \C.
\]
In addition, we have $\alpha=f(1)=\de(a+\1)-\de(a)$, we infer that
\[
\de(a+\la\cdot \1)-\de(a)=\la\{\de(a+\1)-\de(a)\}, \quad \la \in \C,
\]
or
\[
\de(a+\la\cdot \1)-\de(a)=\overline{\la}\{\de(a+\1)-\de(a)\}, \quad \la \in \C.
\]
Let
\[
A_1=\{a \in A;  \de(a+\la\cdot \1)-\de(a)=\la\{\de(a+\1)-\de(a)\}, \quad \la \in \C\}
\]
and 
\[
A_{-1}=\{a \in A;  \de(a+\la\cdot \1)-\de(a)=\overline{\la}\{\de(a+\1)-\de(a)\}, \quad \la \in \C\}.
\]
For any $a\in A$, the map $\lambda\mapsto \de(a+\la\cdot \1)-\de(a)$ is an isometry on $\C$, we have $A=A_1 \cup A_{-1}$.
\begin{lemma}\label{1}
We have $A=A_1$ or $A=A_{-1}$.
\end{lemma}
\begin{proof}
We have proved that $A=A_1 \cup A_{-1}$. We prove that $A_1$ and $A_{-1}$ are closed subsets of $A$. Let $\{a_\alpha \}\subset A_1$ be a net and $a_0 \in A$ such that $a_\alpha\to a_0$. By assumption (b), $\de(a_\alpha)-\de(a_0) \in \mathbb{T}\sigma(a_\alpha-a_0)$. 
Hence $|\de(a_\alpha)-\de(a_0)|\le r(a_\alpha-a_0)$ for the spectral radius $r(\cdot)$. Since $r(\cdot)\le \|\cdot\|$ for the 
original norm $\|\cdot\|$ on $A$,  we get $\de(a_\alpha)-\de(a_0) \to 0$  as  $a_\alpha \to a_0$. In the same way we have that $\de(a_0+\la\cdot \1)-\de(a_\alpha+\la\cdot \1)\to 0$ as $a_\alpha \to a_0$ for any $\lambda\in \C$.
Thus for any $\la \in \C$, we have
\begin{multline*}
|\de(a_0+\la\cdot \1)-\de(a_0)-\la\{\de(a_0+\1)-\de(a_0)\}|\\
=|\de(a_0+\la\cdot \1)-\de(a_0)- \{\de(a_\alpha+\la\cdot \1)-\de(a_\alpha)\}\\ 
+ \la\{\de(a_\alpha+\1)-\de(a_\alpha)\}-\la\{\de(a_0+\1)-\de(a_0)\}|\\
\le|\de(a_0+\la\cdot \1)-\de(a_\alpha+\la\cdot \1)|+|\de(a_0)-\de(a_\alpha)|\\+|\la||\de(a_\alpha+\1)-\de(a_0+\1)|+|\la||\de(a_\alpha)-\de(a_0)|\\
\to 0,
\end{multline*}
as $a_\alpha \to a_0$. This implies that $\de(a_0+\la\cdot \1)-\de(a_0)=\la\{\de(a_0+\1)-\de(a_0)\}$ for any $\la \in \C$. Since $a_0 \in A_1$, we have $A_1$ is closed. We can prove that $A_{-1}$ is also closed in the same way. In addition, suppose that $a\in A_1 \cap A_{-1}$. Then we have for any $\la \in \C$, 
\[
\la\{\de(a+\1)-\de(a)\}=\de(a+\la\cdot \1)-\de(a)=\overline{\la}\{\de(a+\1)-\de(a)\}.
\]
This shows that $\de(a+\1)-\de(a)=0$. On the other hand, we have
\[
\de(a+\1)-\de(a) \in \mathbb{T}\sigma(\1)=\mathbb{T}.
\]
We arrive at a contradiction. Therefore $ A_1 \cap A_{-1} = \emptyset$. Since $A$ is connected, we conclude that $A_1=A$ or $A_{-1}=A$.
\end{proof}
\begin{proof}[Proof of Theorem \ref{gks}]

Lemma \ref{1} shows that 
one of $A=A_1$ and $A=A_{-1}$ occurs.
First we take up the case $A=A_{1}$.  

(i) We consider the case that $A$ is separable.
By the definition of $A_1$, for any $a \in A_1$, we get
\begin{equation}\label{(1)}
\de(a+\la\cdot \1)-\de(a)=\la\{\de(a+\1)-\de(a)\}, \quad \la \in \C.
\end{equation}
 By assumption (b), it follows that 
\[
|\de(a)-\de(b)| \le \|a-b\|, \quad a,b \in A,
\]
which implies that $\de$ is a Lipschitz map. \cite[Theorem 2.3]{ks} (\cite[Theorem 3.4]{lpww}) shows that $\de$ has real differentials except some zero set. We say that $\de$ has a real differential at a point of $a \in A$ if for every $x \in A$ the derivative $\displaystyle \de'_{x}(a)=\lim_{\R \ni r \to 0}\frac{\de(a+rx)-\de(a)}{r}$ exists and the map $(D\de)_a: A \to \C$, defined by $(D\de)_a(x)=\de'_{x}(a)$, is real linear and continuous. (cf.\cite{ma,ks,lpww}.) Since 
\[
\displaystyle \frac{\de(a+rx)-\de(a)}{r} \in \frac{\mathbb{T}\sigma(rx)}{r}= \frac{r\mathbb{T}\sigma(x)}{r}=\mathbb{T}\sigma(x), \quad r \in \mathbb{R}\setminus \{0\},
\]
 we have
\[
(D\de)_{a}(x)=\lim_{\R \ni r \to 0}\frac{\de(a+rx)-\de(a)}{r}\in \mathbb{T}\sigma(x).
\]
As $(D\de)_{a}$ is a real linear, \cite[Lemma 3.3]{lpww} implies that $(D\de)_{a}$ is a complex-linear or conjugate linear. Since $a \in A=A_{1}$, $\de$ satisfies (\ref{(1)}), we have 
\begin{equation*}
\begin{split}
(D\de)_{a}(\1)=\lim_{r \to 0}\frac{\de(a+r\1)-\de(a)}{r}
&=\lim_{r \to 0}\frac{r\{\de(a+\1)-\de(a)\}}{r}\\
&=\de(a+\1)-\de(a) \in \mathbb{T}\sigma(1)=\mathbb{T},
\end{split}
\end{equation*}
and
\begin{equation*}
\begin{split}
(D\de)_{a}(i\1)=\lim_{r \to 0}\frac{\de(a+ri\1)-\de(a)}{r}
&=\lim_{r \to 0}\frac{ri\{\de(a+\1)-\de(a)\}}{r}\\
&=i\{\de(a+\1)-\de(a)\}.
\end{split}
\end{equation*}
It follows that $(D\de)_{a}(i\1)=i(D\de)_{a}(\1)$ and $(D\de)_{a}(\1)\neq 0$. We conclude that 
$(D\de)_{a}$ is complex-linear. We have proved that if $\de$ has a real differential at a point $a \in A=A_1$, then $(D\de)_{a}$ is complex-linear. We conclude that $\de$ is holomorphic in $A$ by applying \cite[Lemma 2.4]{ks}. For $a, b \in A$, we define a map $f_{a,b}: \C \to \C$ by 
\[
f_{a,b}(\la)=\de(\la a+b)-\de(b).
\]
Since $\de$ is holomorphic in A, $f_{a,b}$ is entire. Moreover, we have for any $\la \in \C \setminus \{0\}$ 
\[
\frac{f_{a,b}(\la)}{\la}=\frac{\de(\la a+b)-\de(b)}{\la} \in \frac{\mathbb{T}\sigma(\la a)}{\la}=\frac{\la \mathbb{T} \sigma(a)}{\la}=\mathbb{T}\sigma(a),
\]
and
\[
\left|\frac{f_{a,b}(\la)}{\la}\right|\le \|a\|.
\]
By Liouville's  Theorem, there exists $M \in \C$ such that $f_{a,b}(\la)=\la M$ for all $\la \in \C$. As $M=f_{a,b}(1)=\de(a+b)-\de(b)$, we get
\[
\de(\la a+b)-\de(b)=\la\{\de(a+b)-\de(b)\}, \quad \la \in \C,
\]
and
\begin{equation}\label{(2)}
\de(\la a+b)=\la\{\de(a+b)-\de(b)\}+\de(b), \quad \la \in \C.
\end{equation}
Taking $b=0$ in (\ref{(2)}), we have
\begin{equation}\label{(3)}
\de(\la a)=\la \de(a), \quad \la \in \C,
\end{equation}
by the hypothesis (a). For  any $c, d \in A$, taking $a=\frac{1}{2}(c-d)$, $b=d$ and $\la =2$ in (\ref{(2)}), we get
\begin{equation*}
\begin{split}
\de(c)=\de(2a+b)&=2\{\de(a+b)-\de(b)\}+\de(b)\\
&=2\left \{\de \left (\frac{1}{2}(c+d)\right)-\de(d)\right\}+\de(d)\\
&=2\de\left(\frac{1}{2}(c+d)\right)-\de(d),
\end{split}
\end{equation*}
and 
\begin{equation}\label{(4)}
\de \left (\frac{1}{2}(c+d)\right)=\frac{1}{2}\de(c)+\frac{1}{2}\de(d).
\end{equation}
We conclude that $\de$ is complex-linear by (\ref{(3)}) and (\ref{(4)}). 

(ii) We consider the case $A$ is not separable. For any $a,b \in A$, we can restrict $\de$ to subalgebra $[a,b]$ of $A$ generated by $a$ and $b$. Since $[a,b]$ is separable, we conclude $\de|_{[a,b]}$ is complex-linear. As $a,b$ are chosen arbitrarily, $\de$ is complex-linear too. 

In addition, since $\de(a)=\de(a)-\de(0) \in \mathbb{T}\sigma(a)$, we apply \cite[Proposition 2.2]{lpww} to conclude that $\overline{\de(\1)}\de$ is multiplicative.

Secondly we assume that $A=A_{-1}$. We define the map $\overline{\de}: A \to \C$ by
\[
\overline{\de}(a)=\overline{\de(a)}, \quad a \in A.
\] 
In the case $A=A_{-1}$, $\de$ satisfies for any $a\in A$,
\[
\de(a+\la\cdot \1)-\de(a)=\overline{\la}\{\de(a+\1)-\de(a)\}, \quad \la \in \C.
\]
Thus we have
\[
\overline{\de}(a+\la\cdot \1)-\overline{\de}(a)=\la\{\overline{\de}(a+\1)-\overline{\de}(a)\}, \quad \la \in \C.
\]
Moreover, it is clear that $\overline{\de}(0)=\overline{\de(0)}=0$. Therefore, the map $\overline{\de}:A \to \C$ satisfies the conditions for $\de$ in the case of $A=A_1$. This in turn implies that $\overline{\de}$ is complex-linear and $\overline{\overline{\de}(\1)}\overline{\de}$ is multiplicative. Thus we conclude that $\de$ is conjugate linear and $\overline{\de(\1)}\de$ is multiplicative.
\end{proof}
\section{2-local maps in $\G$}\label{GWC}
In this section $B_j$ is  a unital semisimple commutative Banach algebra with maximal ideal space $M_j$ for $j=1,2$. The Gelfand transform $\hat{\cdot}:B_j \to \widehat{B_j} \subset C(M_j)$ is a continuous isomorphism. Identifying $B_j$ with $\widehat{B_j}$, we consider that $B_j$ is a  subalgebra of $C(M_j)$. 
We say that $f\in B_j$ is unimodular if $|f|=1$ on $M_j$. 
Since $M_j$ is a maximal ideal space and a unimodular element $f$ of $B_j$ has no zeros on $M_j$, $\bar{f}=1/f\in B_j$. 

An interesting generalization of the concept of 2-local maps, that is weak 2-locality, was introduced in \cite{cp,lpww}. We define a {\it pointwise} 2-local map.
\begin{definition}
Let ${\mathcal S}\subset M(B_1,B_2)$. We say $T\in M(B_1,B_2)$ is a pointwise 2-local in ${\mathcal S}$ if for every trio $f,g \in B_1$ and $x\in M_2$ there exists $T_{f,g,x} \in {\mathcal S}$ such that
\[
\text{$\left(T(f)\right)(x)=\left(T_{f,g,x}(f)\right)(x)$ and $\left(T(g)\right)(x)=\left((T_{f,g,x}(g)\right)(x)$}.
\]
\end{definition}
Note that if a map $T$ is 2-local, then $T$ is weak 2-local. If $T$ is weak 2-local, then $T$ is pointwise 2-local. 
We say that $T\in M(B_1,B_2)$ is a pointwise 2-local isometry if $T$ is pointwise 2-local in $\isonly(B_1,B_2)$. Our interest is whether a pointwise 2-local in  $\isonly(B_1,B_2)$ is in fact surjective isometry from $B_1$ onto $B_2$ or not. Simple examples show that a pointwise 2-local isometry need not be a surjection or an isometry. We show three of them. 
\begin{itemize}
\item a map on $C[0,1]$\\ We denote the algebra of all complex-valued continuous functions on $[0,1]$ by $C[0,1]$. The supremum norm $\|\cdot\|_{\infty}$ makes it a Banach algebra. Let $\pi:[0,1]\to [0,1]$ be a continuous function such that $\pi(0)=0$, $\pi(1)=1$ and $0<\pi(x)<1$ for $x \in (0,1)$. Put $T(f)=f\circ\pi$, $f\in C[0,1]$. It is easy to see that $T$ is pointwise 2-local in $\isonly(C[0,1],C[0,1])$ while it is not surjective when $\pi$ is not a homeomorphism. 
\item a map on $C^1[0,1]$\\ We denote the algebra of all continuously differentiable functions defined on the closed unit interval $[0,1]$ by $C^1[0,1]$.
With the norm $\|f\|_{\Sigma}=\|f\|_{\infty}+\|f'\|_{\infty}$ for $f\in C^1[0,1]$, $C^1[0,1]$ is a unital semisimple commutative Banach algebra with maximal ideal space $[0,1]$. Let $T:C^1[0,1]\to C^1[0,1]$ stand $T(f)=\exp(i\cdot)f$, $f\in C^1[0,1]$. By a simple calculation we have $T$ is pointwise 2-local in $\isonly(C^1[0,1],C^1[0,1])$ while $T$ is not an isometry since $\|\1\|_{\Sigma}=1$ and $\|T(\1)\|_{\Sigma}=2$. 
\item a map on the disk algebra $A(\bar{\mathbb D})$. \\
The disk algebra $A(\bar{\mathbb D})$ on $\bar{\mathbb D}$ is the algebra of all continuous functions on $\bar{\mathbb D}$ which are analytic on the open unit disk ${\mathbb D}$. The disk algebra on $\bar{\mathbb D}$ is a uniform algebra on $\bar{\mathbb D}$. It is well known that the maximal ideal space of $A(\bar{\mathbb D})$ is $\bar{\mathbb D}$. Let $\pi_0(z)=z^2$, $z\in \bar{\mathbb D}$. Then the map $T:A(\bar{\mathbb D})\to A(\bar{\mathbb D})$ defined by $T(f)=f\circ\pi_0$, $f\in A(\bar{\mathbb D})$. Trivially $T$ is not surjective, hence $T\not\in \Iso(A(\bar{\mathbb D}), A(\bar{\mathbb D}))$. On the other hand, $T$ is pointwise 2-local in $\isonly(A(\bar{\mathbb D}),A(\bar{\mathbb D}))$. The reason is as follows. Let $f,g\in A(\bar{\mathbb D})$ and $x\in \bar{\mathbb D}$ be arbitrary. If $|x|=1$, then put $\varphi_x(z)=xz$. If $|x|<1$, then it is well known that there is a M\"obius transformation $\varphi_x$ such that $\varphi_x(x)=x^2$ since both of $x$ and $x^2$ is in ${\mathbb D}$. Put $T_{f,g,x}(h)=h\circ \varphi_x$, $h\in A(\bar{\mathbb D})$. We infer by a calculation that $(T(f))(x)=(T_{f,g,x}(f))(x)$ and $(T(g))(x)=(T_{f,g,x}(g))(x)$. Thus $T$ is pointwise 2-local in $\isonly(A(\bar{\mathbb D}), A(\bar{\mathbb D}))$. 
\end{itemize}
It is interesting to point out that a pointwise 2-local isometry is in fact a surjective isometry for some Banach algebra (see subsection \ref{Sd}). A simple example is a pointwise 2-local isometry on the annulus algebra. 
\begin{itemize}
\item
Let $0<r<1$ and $\Omega=\{z: r\le |z|\le 1\}$ be an annulus. Let $A(\Omega)$ be the algebra of all complex-valued continuous functions which is analytic on the interior of $\Omega$. It is well known that $A(\Omega)$ is a uniform algebra on $\Omega$ whose maximal ideal space is homeomorphic to $\Omega$. A pointwise 2-local map in $\isonly(A(\Omega), A(\Omega))$ is a surjective isometry. This can be proved by using the fact that a homeomorphism on $\Omega$ which is analytic on the interior is just a rotation.
\end{itemize}

Recall that for an $\epsilon \in \{\pm 1\}$ and $f\in B_j$, $[f]^{\epsilon}=f$ if $\epsilon =1$ and $[f]^\epsilon=\bar f$ if $\epsilon=-1$. 
Let
\begin{multline*}
\text{$\G=\{T\in M(B_1,B_2);$ there exist a $\beta \in B_2$,}\\
\text{an $\alpha \in B_2 $ with $|\alpha|=1$ on $M_2$,}\\
\text{a continuous map $\pi:M_2\to M_1$,} \\
\text{ and a continuous map $\epsilon: M_2 \to \{\pm1\}$} \\
\text{such that $T(f)=\beta + \alpha [f\circ \pi]^\epsilon$ for every $f\in B_1$\}},
\end{multline*}

Applying Theorem \ref{gks} 
we show that a pointwise 2-local map in $\G$ is also in $\G$.
\begin{theorem}\label{2}
Suppose that $T \in M(B_1,B_2)$ is pointwise 2-local in $\G$. Then there exist a continuous map $\pi:M_2 \to M_1$ and a continuous map $\epsilon: M_2 \to \{\pm1\}$ such that
\begin{equation}\label{thm2}
T(f)=T(0)+(T(\1)-T(0))[f\circ \pi]^\epsilon, \quad f \in B_1,
\end{equation}
where $T(\1)-T(0)$ is a unimodular element in $B_2$. In particular, a pointwise 2-local map in $\G$ is an element in $\G$.
\end{theorem}
\noindent
{\it Proof}.\,\,
Put $T_0=T-T(0)$. 
We infer that $T_0(0)=0$. Since $T$ is pointwise 2-local in $\G$, it is obvious that 
 $T_0$ is also pointwise 2-local in $\G$. Let $x \in M_2$. There exists $\beta_{0,\1,x}, \alpha_{0,\1,x} \in B_2$ with $|\alpha_{0,\1,x}|=1$ on $M_2$, a continuous map $\pi_{0,\1,x}:M_2 \to M_1$ and a continuous map $\epsilon_{0,\1,x}:M_2 \to \{\pm1\}$ such that
\[
T_0(\1)(x)=\beta_{0,\1,x}(x)+\alpha_{0,\1,x}(x)[\1 \circ \pi_{0,\1,x}]^{\epsilon_{0,\1,x}(x)}(x)=\beta_{0,\1,x}(x)+\alpha_{0,\1,x}(x),
\]
and
\[
0=T_0(0)(x)=\beta_{0,\1,x}(x)+\alpha_{0,\1,x}(x)[0 \circ \pi_{0,\1,x}]^{\epsilon_{0,\1,x}(x)}(x)=\beta_{0,\1,x}(x).
\]
It follows that $T_0(\1)(x)=\alpha_{0,\1,x}(x)$. As $x\in M_2$ is arbitrary we have 
\begin{equation}\label{(5)}
|T_0(\1)(x)|=1, \quad x\in M_2.
\end{equation}
Hence $T_0(\1)$ has no zeros on $M_2$, so $\overline{T_0(\1)}=T_0(\1)^{-1}\in B_2$.
 We define $T_1 \in M(B_1,B_2)$ by
\begin{equation}\label{(6)}
T_1=\overline{T_0(\1)}T_0.
\end{equation}
We see that
\begin{equation}\label{(7)}
T_1(0)=\overline{T_0(\1)}T_0(0)=0, \quad  T_1(\1)=\overline{T_0(\1)}T_0(\1)=1
\end{equation}
by (\ref{(5)}). To proceed the proof of Theorem \ref{2}, we need some claims.

{\it Claim} 1. {\it
There exists a map $\pi:M_2 \to M_1$ and a map $\epsilon: M_2 \to \{\pm1\}$ such that
\[
T_1(f)=[f\circ \pi]^\epsilon, \quad f \in B_1.
\]
}
\begin{proof}
By (\ref{(5)}), we infer that $T_1$ is pointwise 2-local in $\G$. Fix $x \in M_2$. We define $\de_x:B_1 \to \C$ by
\begin{equation*}
\de_{x}(f)=(T_1(f))(x), \quad f \in B_1. 
\end{equation*}
As $T_1$ is pointwise 2-local in $\G$, for any $f,g \in B_1$, there exists $T_{f,g,x} \in \G$ such that 
\begin{equation*}
\begin{split}
\de_{x}(f)&=(T_1(f))(x)\\
&=T_{f,g,x}(f)(x)=\beta_{f,g,x}(x)+\alpha_{f,g,x}(x)[f \circ \pi_{f,g,x}]^{\epsilon_{f,g,x}(x)}(x)
\end{split}
\end{equation*}
and
\begin{equation*}
\begin{split}
\de_{x}(g)&=(T_1(g))(x)\\
&=T_{f,g,x}(g)(x)=\beta_{f,g,x}(x)+\alpha_{f,g,x}(x)[g \circ \pi_{f,g,x}]^{\epsilon_{f,g,x}(x)}(x).
\end{split}
\end{equation*}
We infer that
\[
\de_x(f)-\de_{x}(g)=\alpha_{f,g,x}(x)[(f-g) \circ \pi_{f,g,x}]^{\epsilon_{f,g,x}(x)}(x).
\]
If $x \in \epsilon^{-1}_{f,g,x}(1)$, we have
\[
[(f-g) \circ \pi_{f,g,x}]^{\epsilon_{f,g,x}(x)}(x)=(f-g) (\pi_{f,g,x}(x)) \in \sigma(f-g).
\]
If $x \in \epsilon^{-1}_{f,g,x}(-1)$, we have
\[
[(f-g) \circ \pi_{f,g,x}]^{\epsilon_{f,g,x}(x)}(x)=\overline{(f-g) (\pi_{f,g,x}(x))} \in \mathbb{T} \sigma(f-g).
\]
Therefore we get
\[
\de_{x}(f)-\de_{x}(g) \in \mathbb{T} \sigma(f-g), \quad f,g \in B_1.
\]
By (\ref{(7)}), we have $\de_{x}(0)=T_1(0)(x)=0$. Applying Theorem \ref{gks}, we obtain $\de_{x}$ is complex  linear or conjugate linear and $\overline{\de_{x}(\1)}\de_{x}$ is multiplicative. As $\overline{\de_{x}(\1)}=\overline{T_1(\1)(x)}=1$ by (\ref{(7)}), we conclude that $\de_x$ is multiplicative.  In addition $\de_x(\1)=1$ implies that $\de_x \neq 0$. 
Therefore for any $x \in M_2$, one of the following (i) and (ii) occurs:
\def\theenumi{\roman{enumi}}
\begin{enumerate}
\item $\de_x$ is a non-zero multiplicative complex-linear functional,
\item $\de_x$ is a non-zero multiplicative conjugate linear functional.
\end{enumerate}
In the case (i), by Gelfand theory, there exists $\pi(x) \in M_1$ such that 
\[
\de_{x}(f)=f(\pi(x)), \quad f \in B_1.
\]
In the case (ii), $\overline{\de_{x}}$ is non-zero multiplicative complex-linear functional. Thus there exists $\pi(x) \in M_1$ such that 
\[
\overline{\de_{x}}(f)=f(\pi(x)), \quad f \in B_1,
\]
hence
\[
\de_{x}(f)=\overline{f(\pi(x))}, \quad f \in B_1.
\]
Recall that $\de_{x}(f)=(T_1(f))(x)$, we have 
\[
  T_1(f)(x) = \begin{cases}
    f \circ \pi(x), & (\de_x \text{ is complex-linear}) \\
    \overline{f \circ \pi} (x),& (\de_x \text{ is conjugate linear}).
  \end{cases}
\]
We define a map $\epsilon: M_2 \to \{\pm 1\} $ by
\begin{equation}\label{epsilon}
  \epsilon(x) = \begin{cases}
    1, & (\de_x \text{ is complex-linear}) \\
    -1, & (\de_x \text{ is conjugate linear}).
  \end{cases}
\end{equation}
Then we conclude that 
\[
T_1(f)(x)=[f\circ \pi]^{\epsilon(x)}(x), \quad f \in B_1, \  x \in M_2.
\]
\end{proof}

Let
\[
K_1=\{x \in M_2; \de_x \text{ is complex-linear} \}
\]
and
\[
K_{-1}=\{x \in M_2; \de_x \text{ is conjugate linear} \}.
\]
Rewriting \eqref{epsilon}, we have
\[
  \epsilon(x) = \begin{cases}
    1, & (x \in K_1) \\
    -1, & (x \in K_{-1}).
  \end{cases}
\]

{\it Claim} 2.
{\it 
We have $K_1=\{x \in M_2; \de_{x}(i)=i\}$ and $K_{-1}=\{x \in M_2; \de_{x}(i)=-i\}$. In addition $M_2=K_1 \cup K_{-1}$, $K_1 \cap K_{-1} = \emptyset$ and $K_1$ and $K_{-1}$ are closed subset of $M_2$. 
}
\begin{proof}
Since for any $x \in M_2$, $\de_x$ is complex-linear or conjugate linear, it is clear that $M_2=K_1 \cup K_{-1}$. By the definition of $K_1$ and $\de_x(\1)=1$, if $x \in K_1$, then $x \in \{x \in M_2; \de_x(i)=i\}$. Suppose that $x \in \{x \in M_2; \de_x(i)=i\}$. Then $\de_x(i)=i\de_x(\1)$. This implies that $x \in K_1$. We conclude that $K_1=\{x \in M_2; \de_x(i)=i\}$. We can also prove that $K_{-1}=\{x \in M_2; \de_{x}(i)=-i\}$ in the similar argument. Therefore it is easy to see that $K_1 \cap K_{-1} = \emptyset$. Let $\{\x_{\alpha}\} \subset K_1$ be a net with $x_{\alpha} \to x_0 \in M_2$. 
We get
\[
i=\de_{x_{\alpha}}(i)=(T_1(i))(x_{\alpha}) \to (T_1(i))(x_{0})=\de_{x_0}(i).
\]
This implies that $\de_{x_0}(i)=i$ and $x_0 \in K_1$. We have $K_1$ is closed in $M_2$. We also get $K_{-1}$ is closed in the same way.
\end{proof}
Claim 2 shows that $\epsilon: M_2 \to \{\pm1\}$ is continuous.

\vspace{3mm}
{\it Claim} 3.
{\it 
We have $\pi: M_2 \to M_1$ is continuous.
}
\begin{proof}
Let $\{\x_{\alpha}\} \subset M_2$ be a net with $x_{\alpha} \to x_0 \in M_2$. By Claim 2, $K_1$ and $K_{-1}$ are closed and $K_1  \cap K_{-1} = \emptyset$. Thus there is no loss of generality to assume that
\def\theenumi{\roman{enumi}}
\begin{enumerate}
\item  $\{x_{\alpha}\} \subset K_1$ and $x_0 \in K_1$
\item  $\{x_{\alpha}\} \subset K_{-1}$ and $x_0 \in K_{-1}$.
\end{enumerate}
First, we consider the case (i). Then we have 
\[
T_1(f)(x_{\alpha}) \to T_1(f)(x_0), \quad f \in B_1,
\]
hence
\[
(f \circ \pi)(x_{\alpha}) \to (f\circ \pi)(x_0), \quad f \in B_1.
\]
This implies that $\pi(x_\alpha) \to \pi(x_0)$ with the Gelfand topology. For the case (ii), we have 
\[
T_1(f)(x_{\alpha}) \to T_1(f)(x_0), \quad f \in B_1,
\]
and 
\[
\overline{(f \circ \pi)(x_{\alpha})} \to \overline{(f\circ \pi)(x_0)}, \quad f \in B_1.
\]
Thus we get $\pi(x_\alpha) \to \pi(x_0)$ with the Gelfand topology. We conclude that $\pi$ is continuous.
\end{proof}
\begin{proof}[Continuation of Proof of  Theorem \ref{2}]
By (\ref{(6)}), we get $T_0=T_0(\1)T_1$. As $T_0=T-T(0)$ and Claim 1, we have
\begin{equation*}
\begin{split}
T(f)&=T_0(f)+T(0)\\
&=T_0(\1)T_1(f)+T(0)\\
&=T_0(\1)[f\circ \pi]^{\epsilon}+T(0), \quad f \in B_1.
\end{split}
\end{equation*}
Putting $f=\1$, we have $T_0(\1)=T(\1)-T(0)$ and 
\[
T(f)=(T(\1)-T(0))[f \circ \pi]^{\epsilon}+T(0).
\]
In  addition, by (\ref{(5)}), we have $|T_0(\1)|=1$. We obtain that $T_0(\1)=T(\1)-T(0)$ is a unimodular element in $B_2$.
\end{proof}
\begin{remark}
Even though a map $T \in M(B_1,B_2)$ is a 2-local map in $\G$, it is not always the case that $\pi: M_2 \to M_1$ is a homeomorphism. In fact, the map $T_0$ in \cite[Theorem 2.3]{hmto} is a 2-local automorphism, hence 2-local  in $\isonly_{\mathbb{C}}(C(\bar{\mathcal{K}}),C(\bar{\mathcal{K}}))$. On the other hand, the corresponding continuous map is not injective, hence it is not a homeomorphism.
\end{remark}

\begin{cor}\label{homeo}
Suppose that $T \in M(B_1,B_2)$ is a pointwise 2-local in $\G$ and $T$ is injective. Then $\pi(M_2)$ is a uniqueness set for $B_1$, i.e. if $g \in B_1$ and $g=0$ on $\pi(M_2)$, then $g=0$.
\end{cor}
\begin{proof}
Suppose that $g\in B_1$ and  $g=0$ on $\pi(M_2)$. Substituting $g$ in \eqref{thm2}, we get
\[
T(g)=T(0)+(T(\1)-T(0))[g\circ \pi]^\epsilon=T(0)+(T(\1)-T(0))[0]^\epsilon=T(0).
\] 
Since $T$ is injective, we have that $g=0$. Hence $\pi(M_2)$ is a uniqueness set for $B_1$. 
\end{proof}
 
Let
\begin{multline*}
\text{$\operatorname{WC}_{\mathbb C}=\{T\in M(B_1,B_2);$ there exists}\\
\text{an $\alpha \in B_2 $ with $|\alpha|=1$ on $M_2$,}\\
\text{and a continuous map $\pi:M_2\to M_1$} \\
\text{such that $T(f)=\alpha f\circ \pi$ for every $f\in B_1$}
\}.
\end{multline*}
Then $\operatorname{WC}_{\mathbb C}$ is a set of weighted composition operators.  We see that a pointwise 2-local weighted composition operator is a weighted composition operator.
\begin{cor}
Suppose that $T\in M(B_1,B_2)$ is pointwise 2-local in $\operatorname{WC}_{\mathbb C}$. Then $T\in \operatorname{WC}_{\mathbb C}$.
\end{cor}
\begin{proof}
Let $T\in M(B_1,B_2)$ be pointwise 2-local in $\operatorname{WC}_{\mathbb C}$. Since $\operatorname{WC}_{\mathbb C}\subset \G$, we see by Theorem \ref{2} that there exist a continuous map $\pi:M_2 \to M_1$ and a continuous map $\epsilon: M_2 \to \{\pm1\}$ such that
\begin{equation}\label{thm22}
T(f)=T(0)+(T(\1)-T(0))[f\circ \pi]^\epsilon, \quad f \in B_1,
\end{equation}
where $T(\1)-T(0)$ is a unimodular element in $B_2$.
Since any map in $\operatorname{WC}_{\mathbb C}$ is complex-linear, we infer by a simple calculation that $T(0)=0$ and $T$ is homogeneous with respect to complex scalar. We see by \eqref{thm22} that
\[
T(f)=T(\1)f\circ \pi, \quad f\in B_1,
\]
where $T(\1)$ is a unimodular function. Thus $T\in \operatorname{WC}_{\mathbb C}$.
\end{proof}

\section{Applications}
In this section we study 2-local isometries on several function spaces by applying Theorem \ref{2}.
\subsection{Uniform algebras}
Let $X$ be a compact Hausdorff space. The algebra of all complex-valued continuous functions on $X$ is denoted by $C(X)$, which is a Banach algebra with respect to the supremum norm $\|\cdot\|_{\infty}$ on $X$. 
We say that $A$ is a uniform algebra on $X$ if $A$ is a uniformly closed  subalgebra of $C(X)$ which contains constant functions and separates the points of $X$. As the Gelfand transformation on a uniform algebra is an isometric isomorphism, a uniform algebra is isometrically isomorphic to its Gelfand transform. We may suppose that $X$ is a subset of the maximal ideal space $M_A$, and $A$ is a uniform algebra on $M_A$. 
The Banach algebra $C(X)$ is a uniform algebra on $X$ whose maximal ideal space is $X$. 
By Theorem 2.1 and Corollary 3.4 in \cite{hm} we have the following. Note that we denote the maximal ideal space of a uniform algebra $A_j$ by $M_j$ for $j=1,2$.
\begin{theorem}\label{uniform}
Let $A_j$ be a uniform algebra on a compact Hausdorff space $X_j$ for $j=1,2$. Suppose that $U:A_1\to A_2$ is a surjective isometry from $A_1$ onto $A_2$. Then there exists a homeomorphism $\pi:M_2\to M_1$, an $\alpha\in A_2$ with $|\alpha|=1$ on $M_2$, and a continuous map $\epsilon:M_2\to \{\pm 1\}$ such that
\begin{equation}\label{abc}
U(f)=U(0)+\alpha[f\circ\pi]^{\epsilon}, \quad f\in A_1.
\end{equation}
\end{theorem}
If $A_j=C(X_j)$, the map $U$ defined by \eqref{abc} is a surjective isometry from $C(X_1)$ onto $C(X_2)$. 

By Theorem \ref{uniform} we see that 
\[
\isonly(A_1,A_2)\subset \G
\]
for uniform algebras $A_1$ and $A_2$. 
A direct consequence of Theorem \ref{2} we have Corollary \ref{uniform2}, which is a generalization of Theorem 3.10 of \cite{lpww}.
\begin{cor}\label{uniform2}
Let $A_j$ be a uniform algebra on a compact Hausdorff space $X_j$ for $j=1,2$. Suppose that $T\in M(A_1,A_2)$ is pointwise 2-local in $\isonly(A_1,A_2)$. 
Then there exist a continuous map $\pi:M_2 \to M_1$ and a continuous map $\epsilon: M_2 \to \{\pm1\}$ such that
\[
T(f)=T(0)+(T(\1)-T(0))[f\circ \pi]^\epsilon, \quad f \in A_1,
\]
where $T(\1)-T(0)$ is a unimodular function.
\end{cor}
  We have the following. 
\begin{cor}\label{C(X)}
Let $X_j$ be a first countable compact Hausdorff space for $j=1,2$. Suppose that $T \in M(C(X_1),C(X_2))$ is 2-local in $\operatorname{Iso}(C(X_1), C(X_2))$. Then we have  $T \in \operatorname{Iso}(C(X_1),  C(X_2))$.
\end{cor}
\begin{proof}
Let $T$ be a 2-local in $\operatorname{Iso}(C(X_1), C(X_2))$.  By Corollary $\ref{uniform2}$, there exist a continuous map $\pi:X_2 \to X_1$ and a continuous map $\epsilon: X_2 \to \{\pm1\}$ such that
\begin{equation}\label{112233}
T(f)=T(0)+(T(\1)-T(0))[f\circ \pi]^\epsilon, \quad f \in C(X_1).
\end{equation}
We prove $\pi$ is an injection. 
Suppose that $y_1, y_2 \in X_2$ such that $\pi(y_1)=\pi(y_2)=x \in X_1$. Since $X_1$ is first countable there exists $g \in C(X_1)$ such that  $g^{-1}(0)=\{x\}$. Since $T_1=\overline{T_0(\1)}T_0$ for $T_0=T-T(0)$ is 2-local in $\operatorname{Iso}(C(X_1), C(X_2))$, we have 
\begin{equation*}
\begin{split}
0=T_1(0)&=T_{0,g}(0)\\
&=\beta_{0,g}+\alpha_{0,g}[0 \circ \pi_{0,g}]^{\epsilon_{0,g}}=\beta_{0,g},
\end{split}
\end{equation*}
and
\begin{equation*}
\begin{split}
T_1(g)&=T_{0,g}(g)\\
&=\beta_{0,g}+\alpha_{0,g}[g \circ \pi_{0,g}]^{\epsilon_{0,g}}.
\end{split}
\end{equation*}
Hence we see that
\[
T_1(g)=\alpha_{0,g}[g \circ \pi_{0,g}]^{\epsilon_{0,g}}.
\]
We have 
\[
(T_1(g))^{-1}(0)=(g \circ \pi_{0,g})^{-1}(0)=\pi_{0,g}^{-1}(x).
\]
Since $\pi_{0,g}$ is homeomorphism, the set $\pi_{0,g}^{-1}(x)$ is a singleton. Moreover applying \eqref{112233} we have
\[
T_1(g)=[g\circ \pi]^{\epsilon}.
\]
Thus we obtain 
\[
(T_1(g))^{-1}(0)=(g\circ \pi)^{-1}(0)=\pi^{-1}(x) \ni \{y_1,y_2\}.
\]
As we have already proved that the set $(T_1(g))^{-1}(0)=\pi_{0,g}^{-1}(x)$ is a singleton, we infer that $y_1=y_2$. Thus $\pi$ is injective.
Since $T$ is a 2-local isometry, $T$ is an isometry by the definition of a 2-local isometry. Hence $T$ is injective. 
By Corollary \ref{homeo}, $\pi(X_2)$ is a uniqueness set for $C(X_1)$, which is $X_1$ itself. As $X_j$ is compact Hausdorff space, we infer that $\pi$ is a homeomorphism.  It follows that $T \in \operatorname{Iso}(C(X_1), C(X_2))$
\end{proof}
 Corollary \ref{C(X)} gives an affirmative answer to the problem mentioned by Moln\'ar. Mori proved the same statement in \cite[Theorem 4.6]{hamo} by a different arugument. 


Next we consider the disk algebra. 
Let $\bar{\mathbb D}$ be the closed unit disk. 
\begin{cor}\label{disk}
Suppose that $U$ is a surjective isometry from the disk algebra $A(\bar{\mathbb D})$ onto itself. Then there exists a M\"obius transformation 
$\varphi$ on $\bar{\mathbb D}$ and a unimodular constant $\alpha$ such that 
\[
U(f)=U(0)+\alpha f\circ\varphi,  \quad f\in A(\bar{\mathbb D})
\]
or
\[
U(f)=U(0)+\alpha \overline{f\circ\bar{\varphi}},  \quad f\in A(\bar{\mathbb D}).
\]
Conversely if one of the above equations holds, then $U$ is a surjective isometry from the disk algebra onto itself.
\end{cor}
\begin{proof}
Applying Theorem \ref{uniform} we have a homeomorphism $\pi:\bar{\mathbb D}\to \bar{\mathbb D}$, a unimodular function $\alpha\in A(\bar{\mathbb D})$, and a continuous map $\epsilon:\bar{\mathbb D}\to \{\pm 1\}$ such that 
\begin{equation}\label{*}
U(f)=U(0)+ \alpha [f\circ\pi]^{\epsilon},\quad f\in A(\bar{\mathbb D}).
\end{equation}
Due to the maximum modulus principle for analytic functions, $\alpha$ is a unimodular constant. Since $\bar{\mathbb D}$ is connected, $\epsilon=1$ on $\bar{\mathbb D}$, or $\epsilon=-1$ on $\bar{\mathbb D}$. Letting $f=\operatorname{Id}$, the identity function, in \eqref{*} we have
\begin{equation}\label{*1}
\text{$\bar{\alpha}(U(\id)-U(0))=\pi$ if $\epsilon =1$},
\end{equation}
\begin{equation}\label{*2}
\text{$\bar{\alpha}(U(\id)-U(0))=\bar{\pi}$ if $\epsilon =-1$}.
\end{equation}
Suppose that $\epsilon=1$. Then $\pi$ is analytic on ${\mathbb D}$ by \eqref{*1}. As $\pi$ is a homeomorphism, we conclude that $\pi$ is a M\"obius transformation. In the same way, $\bar{\pi}$ is a M\"obius transformation if $\epsilon=-1$. Letting $\varphi=\pi$ if $\epsilon=1$, and $\varphi=\bar{\pi}$ if $\epsilon=-1$, $\varphi$ is a M\"obius transformation. It follows that 
\[
U(f)=U(0)+\alpha f\circ \varphi,\quad f\in A(\bar{\mathbb D})
\]
if $\epsilon=1$ and 
\[
U(f)=U(0)+\alpha \overline{f\circ \bar{\varphi}},\quad f\in A(\bar{\mathbb D})
\]
if $\epsilon=-1$. 

The converse statement is trivial.
\end{proof}
By Corollary \ref{disk} we see that 
\[
\isonly(A(\bar{\mathbb D}), A(\bar{\mathbb D}))\subset \G
\]
for the disk algebra $A(\bar{\mathbb D})$. 
\begin{cor}\label{diskalgebra}
Suppose that $T\in M(A(\bar{\mathbb D}),A(\bar{\mathbb D}))$ is 2-local in $\isonly(A(\bar{\mathbb D}),A(\bar{\mathbb D}))$. Then $T\in \isonly(A(\bar{\mathbb D}),A(\bar{\mathbb D}))$.
\end{cor}
\begin{proof}
Corollary \ref{uniform2} asserts that 
there exist a continuous map $\pi:\bar{\mathbb D} \to \bar{\mathbb D}$ and a continuous map $\epsilon: \bar{\mathbb D} \to \{\pm1\}$ such that
\begin{equation}\label{disk*}
T(f)=T(0)+(T(\1)-T(0))[f\circ \pi]^\epsilon, \quad f \in A(\bar{\mathbb D}),
\end{equation}
where $T(\1)-T(0)$ is a unimodular function. 
By the same way as the proof of Corollary \ref{disk} we see that 
 $T(\1)-T(0)$ is a unimodular constant. 
We also see that $\epsilon =1$ on $\bar{\mathbb D}$ or $\epsilon=-1$ on $\bar{\mathbb D}$ because $\bar{\mathbb D}$ is connected and $\epsilon$ is continuous. Letting $f=\id$ in \eqref{disk*}, we have that $\pi$ is analytic on ${\mathbb D}$ if $\epsilon=1$, and $\bar{\pi}$ is analytic on ${\mathbb D}$ if $\epsilon =-1$. Put $\varphi=\pi$ if $\epsilon =1$ and $\varphi=\bar{\pi}$ if $\epsilon=-1$. Put $T_1=\overline{T(\1)-T(0)}(T-T(0))$. Then 
\[
T_1(f)=f\circ \varphi,\quad f\in A(\bar{\mathbb D})
\]
if $\epsilon=1$, and 
\[
T_1(f)=\overline{f\circ \bar{\varphi}},\quad f\in A(\bar{\mathbb D})
\]
if $\epsilon=-1$.
Since $T_1$ is 2-local in $\isonly(A(\bar{\mathbb D}),A(\bar{\mathbb D}))$, there exists a M\"obius transform $\varphi_0$, $u\in A(\bar{\mathbb D})$, and a unimodular constant $\alpha$ such that
\[
\text{$\varphi=T_1(\id)=u+\alpha\varphi_0$ and $0=T_1(0)=u$}.
\]
It follows that $\varphi=\alpha\varphi_0$. As $|\alpha|=1$, we infer that $\varphi$ is a M\"obius transformation on $\bar{\mathbb D}$. We infer by Corollary \ref{disk} that $T\in \isonly(A(\bar{\mathbb D}),A(\bar{\mathbb D}))$.
\end{proof}
\subsection{Lipschitz algebras} Let $(X_j,d)$ be a compact metric space for $j=1,2$. Let
\[
\Lip(X_j)=\left\{f\in C(X_j):L(f)=\sup_{x\ne y}\frac{|f(x)-f(y)|}{d(x,y)}<\infty\right\}.
\]
We say that $L(f)$ is the Lipschitz constant for $f$. With 
the norm $\|f\|_{\Sigma}=\|f\|_{\infty}+L(f)$ for $f\in \Lip(X_j)$, the algebra 
 $\Lip(X_j)$ is a unital semisimple commutative Banach algebra. In addition, maximal ideal space of $\Lip(X_j)$ can be identified with $X_j$.
\begin{cor}\label{L}
Let $\|\cdot\|_j$ be any  norm on $\Lip(X_j)$. We do not assume that $\|\cdot\|_j$ is complete. Suppose that
\begin{multline}\label{lipiso}
\text{$\operatorname{Iso}((\Lip(X_1),\|\cdot\|_1), (\Lip(X_2),\|\cdot\|_{2}))$}\\
\text{$=\{T\in M(\Lip(X_1), \Lip(X_2));$}\\
\text{there exist  $\beta \in \Lip(X_2)$, $\alpha \in \mathbb{T}$,}\\
\text{a surjective isometry $\pi: X_2 \to X_1$, and $\epsilon=\pm1$} \\
\text{such that $T(f)=\beta + \alpha [f\circ \pi]^\epsilon$ for every $f \in \Lip(X_1)$\}}.
\end{multline}
Suppose that $T \in M((\Lip(X_1),\|\cdot\|_1), (\Lip(X_2),\|\cdot\|_{2}))$  is 2-local in $\operatorname{Iso}((\Lip(X_1),\|\cdot\|_1), (\Lip(X_2),\|\cdot\|_{2}))$.
Then  $T \in \operatorname{Iso}((\Lip(X_1),\|\cdot\|_1), (\Lip(X_2),\|\cdot\|_{2}))$.
\end{cor}
\begin{proof}
Suppose that  $T$ is 2-local in $\operatorname{Iso}((\Lip(X_1),\|\cdot\|_1), (\Lip(X_2),\|\cdot\|_{2}))$. The equality \eqref{lipiso} implies that $\operatorname{Iso}((\Lip(X_1),\|\cdot\|_1), (\Lip(X_2),\|\cdot\|_{2})) \subset \G$. Applying Theorem \ref{2}, there exists a continuous map $\pi: X_2 \to X_1$ and a continuous map $\epsilon: X_2 \to \{\pm1\}$ such that
\begin{equation}\label{LIPISO}
T(f)=T(0)+(T(\1)-T(0))[f\circ \pi]^\epsilon, \quad f \in\Lip(X_1).
\end{equation}
Recall that $T_1=\overline{T_0(\1)}T_0$ for $T_0=T-T(0)$. Since $T_0$ is 2-local, we have 
\[
T_0(\1)=\beta_{0,\1}+\alpha_{0,\1}[\1 \circ \pi_{0,\1}]^{\epsilon_{0,\1}},
\]
and
\[
0=T_0(0)=\beta_{0,\1}+\alpha_{0,\1}[0 \circ \pi_{0,\1}]^{\epsilon_{0,\1}}=\beta_{0,\1}.
\]
It follows that $T(\1)-T(0)=T_0(\1)$ is a unimodular constant. Thus $T_1=\overline{T_0(\1)}T_0$ is 2-local in $\operatorname{Iso}((\Lip(X_1),\|\cdot\|_1), (\Lip(X_2),\|\cdot\|_{2}))$. We get 
\begin{equation*}
\begin{split}
0=T_1(0)&=T_{0,i}(0)\\
&=\beta_{0,i}+\alpha_{0,i}[0 \circ \pi_{0,i}]^{\epsilon_{0,i}}=\beta_{0,i},
\end{split}
\end{equation*}
and
\begin{equation*}
\begin{split}
T_1(i)&=T_{0,i}(i)\\
&=\beta_{0,i}+\alpha_{0,i}[i \circ \pi_{0,i}]^{\epsilon_{0,i}}.
\end{split}
\end{equation*}
We get
\[
T_1(i)=\alpha_{0,i}[i \circ \pi_{0,i}]^{\epsilon_{0,i}}.
\]
Since $\alpha_{0,i}$ is a unimodular constant and $\epsilon_{0,i}=\pm1$,  so we obtain $T_1(i)$ is a constant. Moreover applying \eqref{LIPISO}, we have
\[
T_1(i)=[i\circ \pi]^{\epsilon}.
\]
Thus we conclude that $\epsilon=1$ or $\epsilon=-1$. As $T$ is a 2-local isometry, $T$ is an isometry, hence $T$ is injective. Corollary \ref{homeo} asserts that $\pi(X_2)$ is a uniqueness set for $\Lip(X_1)$. Thus we have $\pi(X_2)=X_1$. This implies that $\pi$ is surjective. Finally we shall prove that $\pi$ is an isometry. 
Let $x_0 \in X_2$.  We define a Lipschitz function ${g}$ on $X_1$ by 
\[
g(x)=d(x,\pi(x_0)),\quad x\in X_1.
\]
As $T_1$ is  2-local in $\operatorname{Iso}((\Lip(X_1),\|\cdot\|_1), (\Lip(X_2),\|\cdot\|_{2}))$, there exists $\alpha_{0,g} \in \mathbb{T}$ and $\pi_{0,g}: X_2 \to X_1$ is a surjective isometry such that
\begin{equation*}
\begin{split}
0=T_1(0)&=T_{0,g}(0)\\
&=\beta_{0,g}+\alpha_{0,g}[0 \circ \pi_{0,g}]^{\epsilon_{0,g}}=\beta_{0,g},
\end{split}
\end{equation*}
and
\begin{equation*}
\begin{split}
T_1(g)&=T_{0,g}(g)\\
&=\beta_{0,g}+\alpha_{0,g}[g \circ \pi_{0,g}]^{\epsilon_{0,g}}=\beta_{0,g}+\alpha_{0,g} g \circ \pi_{0,g},
\end{split}
\end{equation*}
because $g$ is a real-valued function. If follows that 
\[
(T_1(g))(z)=\alpha_{0,g}g(\pi_{0,g}(z)), \quad z \in X_2.
\]
By \eqref{LIPISO}, for any $z \in X_2$
\begin{multline}\label{33}
d(\pi(z),\pi(x_0))=[g(\pi(z))]^{\epsilon}\\=(T_1(g))(z)=\alpha_{0,g}g(\pi_{0,g}(z))=\alpha_{0,g}d(\pi_{0,g}(z),\pi(x_0)).
\end{multline}
We may suppose that $X_1$ is not a singleton. (Suppose that it is so. Then $X_2$ is a singleton since $\pi_{0,g}$ is a surjective isometry. Then $\pi$ is automatically surjective isometry.) Hence there exists $z_0\in X_2$ such that $d(\pi_{0,g}(z_0),\pi(x_0))\ne 0$. By \eqref{33} with $z=z_0$ we have 
\[
\displaystyle \alpha_{0,g}=\frac{d(\pi(z_0),\pi(x_0))}{d(\pi_{0,g}(z_0),\pi(x_0))} \ge 0,
\]
we obtain $\alpha_{0,g}=1$.  Hence by \eqref{33} we have
\begin{equation}\label{3333}
d(\pi(z),\pi(x_0))=d(\pi_{0,g}(z),\pi(x_0)), \quad z\in X_2.
\end{equation}
Putting $z=x_0$ in \eqref{3333}, we have
\[
0=d(\pi(x_0),\pi(x_0))=d(\pi_{0,g}(x_0),\pi(x_0)).
\]
It follows $\pi_{0,g}(x_0)=\pi(x_0)$. By (\ref{3333})
\[
d(\pi(z),\pi(x_0))=d(\pi_{0,g}(z),\pi(x_0))=d(\pi_{0,g}(z),\pi_{0,g}(x_0))=d(z,x_0)
\]
since $\pi_{0,g}$ is an isometry. As $z$ and $x_0$ are arbitrary, we conclude that $\pi$ is an isometry. This completes the proof.
\end{proof}
For an arbitrary compact metric space $X_j$ 
for $j=1,2$,  \cite[Theorem 6]{haoi} shows that $\operatorname{Iso}((\Lip(X_1),\|\cdot\|_{\Sigma}), (\Lip(X_2),\|\cdot\|_{\Sigma}))$ fulfills the condition of Corollary \ref{L}. Thus we have the following.
\begin{cor}\label{Lip2}
Suppose that $T \in M(\Lip(X_1), \Lip(X_2))$ is 2-local in $\operatorname{Iso}((\Lip(X_1),\|\cdot\|_{\Sigma}), (\Lip(X_2),\|\cdot\|_{\Sigma}))$. 
Then  $T \in \operatorname{Iso}((\Lip(X_1),\|\cdot\|_{\Sigma}), (\Lip(X_2),\|\cdot\|_{\Sigma}))$.
\end{cor}

Corollary \ref{Lip2} generalizes Theorem 8 in \cite{haoi}, where the case $X_1=X_2=[0,1]$
 is proved.
\subsection{The algebra of continuously differentiable functions} 
We denote the algebra of all continuously differentiable functions by $C^{1}([0,1])$. It is a unital semisimple commutative Banach algebra with the norm $\|\cdot\|_{\Sigma}$ defined by
\[
\|f\|_{\Sigma}=\|f\|_{\infty}+\|f'\|_{\infty}, \quad f \in C^{1}([0,1]).
\]
The maximal ideal space of $C^{1}([0,1])$ is homeomorphic to $[0,1]$. We have the following corollary.
\begin{cor}\label{C12}
Let $\|\cdot\|_j$ be any norm on $C^1([0,1])$ for $j=1,2$. We do not assume that $\|\cdot\|_j$ is complete. Suppose that
\begin{multline}\label{1344}
\text{$\operatorname{Iso}((C^{1}([0,1]),\|\cdot\|_1), (C^{1}([0,1]),\|\cdot\|_{2}))$}\\
\text{$=\{T\in M(C^{1}([0,1]), C^{1}([0,1]));$}\\
\text{there exist  $\beta \in C^{1}([0,1])$, $\alpha \in \mathbb{T}$,}\\
\text{$\pi=\id$ or $\pi=1-\id$ and $\epsilon=\pm1$} \\
\text{such that $T(f)=\beta + \alpha [f\circ \pi]^\epsilon$ for every $f \in C^{1}([0,1])$\}}.
\end{multline}
Suppose that $T \in M(C^{1}([0,1]), C^{1}([0,1]))$ is 2-local in $\operatorname{Iso}((C^{1}([0,1]),\|\cdot\|_1), (C^{1}([0,1]),\|\cdot\|_{2}))$.
Then $T \in \operatorname{Iso}((C^{1}([0,1]),\|\cdot\|_1), (C^{1}([0,1]),\|\cdot\|_{2}))$.
\end{cor}
\begin{proof}
Let $T$ be 2-local in $\operatorname{Iso}((C^{1}([0,1]),\|\cdot\|_1), (C^{1}([0,1]),\|\cdot\|_{2}))$. By \eqref{1344}, $\operatorname{Iso}((C^{1}([0,1]),\|\cdot\|_1), (C^{1}([0,1]),\|\cdot\|_{2}))\subset \G$. Theorem \ref{2} asserts that there exists a continuous map $\pi: [0,1] \to [0,1]$ and a continuous map $\epsilon: [0,1] \to \{\pm1\}$ such that
\begin{equation}\label{5555}
T(f)=T(0)+(T(\1)-T(0))[f\circ \pi]^\epsilon, \quad f \in C^{1}([0,1]).
\end{equation}
Since $\epsilon: [0,1] \to \{\pm1\}$ is continuous and $[0,1]$ is connected, we conclude that $\epsilon=\pm1$. 
As $T$ is a 2-local isometry, we get $T$ is an isometry. This implies that $T$ is injective. Corollary \ref{homeo} asserts that $\pi([0,1])$ is a uniqueness set for $C^{1}([0,1])$, which is $[0,1]$. Thus we have $\pi$ is surjective. To complete the proof we shall prove that $\pi$ is an isometry. Let $x_0 \in [0,1]$. We define the function $g(x)=x-\pi(x_0) \in C^{1}[0,1]$. Define $T_1=\overline{T_0(\1)}T_0$ for $T_0=T-T(0)$. It is easy to see that $T_0$ is 2-local in $\operatorname{Iso}((C^{1}([0,1]),\|\cdot\|_1), (C^{1}([0,1]),\|\cdot\|_{2}))$, we have 
\[
T_0(\1)=\beta_{0,\1}+\alpha_{0,\1}[\1 \circ \pi_{0,\1}]^{\epsilon_{0,\1}},
\]
and
\[
0=T_0(0)=\beta_{0,\1}+\alpha_{0,\1}[0 \circ \pi_{0,\1}]^{\epsilon_{0,\1}}=\beta_{0,\1}.
\]
It follows that $T(\1)-T(0)=T_0(\1)$ is a unimodular constant. 
We have $T_1=\overline{T_0(\1)}T_0$ is 2-local in $\operatorname{Iso}((C^{1}([0,1]),\|\cdot\|_1), (C^{1}([0,1]),\|\cdot\|_{2}))$.  Hence we get
\begin{equation*}
\begin{split}
0=T_1(0)&=T_{0,g}(0)\\
&=\beta_{0,g}+\alpha_{0,g}[0 \circ \pi_{0,g}]^{\epsilon_{0,g}}=\beta_{0,g},
\end{split}
\end{equation*}
and
\begin{equation*}
\begin{split}
T_1(g)&=T_{0,g}(g)\\
&=\beta_{0,g}+\alpha_{0,g}[g \circ \pi_{0,g}]^{\epsilon_{0,g}}.
\end{split}
\end{equation*}
It follows that 
\[
(T_1(g))(z)=\alpha_{0,g}[g \circ \pi_{0,g}]^{\epsilon_{0,g}}(z)=\alpha_{0,g}[g(\pi_{0,g}(z))]^{\epsilon_{0,g}}, \quad z \in [0,1].
\]
Thus by \eqref{5555}, we have for any $z \in [0,1]$ that
\begin{multline*}
[\pi(z)-\pi(x_0)]^{\epsilon}=[g(\pi(z))]^{\epsilon}=(T_1(g))(z)\\
=\alpha_{0,g}[g(\pi_{0,g}(z))]^{\epsilon_{0,g}}=\alpha_{0,g}[\pi_{0,g}(z)-\pi(x_0)]^{\epsilon_{0,g}},\end{multline*}
where $\alpha_{0,g} \in \mathbb{T}$ and $\pi_{0,g}=\id$ or $\pi_{0,g}=1-\id$. Putting $z=x_0$, we have 
\[
0=[\pi(x_0)-\pi(x_0)]^{\epsilon}=\alpha_{0,g}[\pi_{0,g}(x_0)-\pi(x_0)]^{\epsilon_{0,g}}.
\]
It follows that $\pi_{0,g}(x_0)=\pi(x_0)$. Thus we have 
\begin{multline*}
[\pi(z)-\pi(x_0)]^{\epsilon}
=\alpha_{0,g}[\pi_{0,g}(z)-\pi(x_0)]^{\epsilon_{0,g}}=\alpha_{0,g}[\pi_{0,g}(z)-\pi_{0,g}(x_0)]^{\epsilon_{0,g}},
\end{multline*}
and
\[
|\pi(z)-\pi(x_0)|=|\pi_{0,g}(z)-\pi_{0,g}(x_0)|=|z-x_0|.
\]
As $z$ and $x_0$ are arbitrary, we conclude that $\pi$ is an isometry. This completes the proof. 
\end{proof}

In \cite{kkm,mt}, they gave the characterization for surjective isometries on $C^{1}([0,1])$ with respect to various norms. There are many norms with which the groups of surjective isometries on $C^{1}([0,1])$ fulfills the condition of Corollary \ref{C12}. We present one of them. 
\begin{cor}\label{c101}
Suppose that $T \in M((C^{1}([0,1]),\|\cdot\|_{\Sigma}), (C^{1}([0,1]),\|\cdot\|_{\Sigma}))$ and $T$ is 2-local in $\operatorname{Iso}((C^{1}([0,1]),\|\cdot\|_{\Sigma}), (C^{1}([0,1]),\|\cdot\|_{\Sigma}))$. We conclude that  $T \in \operatorname{Iso}((C^{1}([0,1]),\|\cdot\|_{\Sigma}), (C^{1}([0,1]),\|\cdot\|_{\Sigma}))$.
\end{cor}
Corollary \ref{c101} has been shown in \cite[Theorem 9]{haoi} in a different way.
\subsection{The algebra $\Sd$}\label{Sd}
Let 
\[
\Sd=\{f \in H(\mathbb{D}); f' \in H^{\infty}(\mathbb{D})\},
\]
where $H(\mathbb{D})$ is the linear space of all analytic functions on $\mathbb{D}$ and $H^{\infty}(\mathbb{D})$ is the algebra of all bounded analytic functions on $\mathbb{D}$. The algebra $\Sd$  equipped with the norm $\|f\|_{\Sigma}=\sup_{z \in \mathbb{D}}|f(z)|+\sup_{w \in \mathbb{D}}|f'(w)|$ for $f \in \Sd$ is a unital semisimple commutative Banach algebra. As is described in \cite{mi}, $\Sd$ coincides with the space of all Lipschitz functions in the linear space of all analytic functions on ${\mathbb D}$ and each $f\in \Sd$ is continuously extended to the closed unit disk $\bar{\mathbb D}$. Hence we may suppose that $\Sd$ is a unital subalgebra of the disk algebra on $\bar{\mathbb D}$. Trivially all analytic polynomials are in $\Sd$. 
\begin{theorem}
The maximal ideal space $M_\infty$ of $\Sd$ is homeomorphic to the closed unit disk $\bar{\mathbb D}$.
\end{theorem}
\begin{proof}
For each $p\in \bar{\mathbb D}$, the point evaluation on $\Sd$ which takes the value at $p$ is a nontrivial complex homomorphism. Hence we may suppose that $\bar{\mathbb D}\subset M_\infty$. To prove $\bar{\mathbb D}=M_\infty$, suppose that $f_1,\dots, f_n$ be an arbitrary finite number of functions in $\Sd$ such that
\[
\text{$\sum_{j=1}^n|f_j|>0$ on $\bar{\mathbb D}$}.
\]
If we prove that there exist the same number of $g_1,\dots, g_n\in \Sd$ such that 
\[
\sum_{j=1}^nf_jg_j=1,
\]
then a general result assures that $\bar{\mathbb D}=M_\infty$. We prove the existence of such $g_1,\dots, g_n \in \Sd$. It is well known that  the maximal ideal space of the disk algebra $A(\bar{\mathbb D})$ is $\bar{\mathbb D}$. As $f_1,\dots, f_n\in \Sd\subset A(\bar{\mathbb D})$, there exists $h_1,\dots, h_n\in A(\bar{\mathbb D})$ such that 
\[
\sum_{j=1}^nf_jh_j=1.
\]
As functions in $A(\bar{\mathbb D})$ are uniformly approximated by analytic polynomials, there exists a sequence of polynomials $\{p_m^{(j)}\}_{m=1}^\infty$ such that $\|p_m^{(j)}-h_j\|_{\infty}\to 0$ as $m\to \infty$ for every $j=1,\dots, n$. Hence for sufficiently large $m_0$ we have
\[
\left\|1-\sum_{j=1}^nf_jp_{m_0}^{(j)}\right\|<1/2.
\]
In particular, $\sum_{j=1}^nf_jp_{m_0}^{(j)}$ has no zeros on $\bar{\mathbb D}$. Then  $1/\sum_{j=1}^nf_jp_{m_0}^{(j)}\in \Sd$. Put $g_j=p_{m_0}^{(j)}/\sum_{j=1}^nf_jp_{m_0}^{(j)}$ for $j=1,\dots, n$. Then $g_j\in \Sd$ and $\sum_{j=1}^nf_jg_j=1$ by a simple calculation. It follows that $\bar{\mathbb D}=M_\infty$.
\end{proof}

Miura \cite[Theorem1]{mi} showed the form of the surjective isometry on $\Sd$. 
\begin{theorem}[Miura \cite{mi}]\label{miura}
Suppose that $U:\Sd \to \Sd$ is a surjective isometry with respect to the norm $\|\cdot\|_{\Sigma}$. Then there exists unimodular constants $\alpha,\lambda\in {\mathbb C}$ such that 
\[
U(f)=U(0)+\alpha f(\lambda\cdot),\quad f\in \Sd
\]
or
\[
U(f)=U(0)+\alpha \overline{f(\overline{\lambda\cdot}}),\quad f\in \Sd.
\]
Conversely, each of the above form is a surjective isometry from $\Sd$ onto $\Sd$.
\end{theorem}
As we stated in the beginning of Section \ref{GWC}, for some Banach algebras $B_j$, a pointwise  2-local map in $\Iso(B_1,B_2)$ is not always a  surjective isometry. But applying Theorem \ref{2} and  Theorem \ref{miura} we deduce that a pointwise 2-local map in $\Iso(\Sd,\Sd)$ is always a surjective isometry.

\begin{cor}\label{S}
Suppose that $T \in M(\Sd,\Sd)$ is pointwise 2-local in $\Iso(\Sd,\Sd)$. Then  $T \in \Iso(\Sd,\Sd)$.
\end{cor}
\begin{proof}
Let $T \in M(\Sd,\Sd)$ be a pointwise 2-local map in $\Iso(\Sd,\Sd)$. By Theorem \ref{miura} $\Iso(\Sd,\Sd)\subset \G$. Then Theorem \ref{2} asserts that there exist a continuous map $\pi:\bar{\mathbb D}\to \bar{\mathbb D}$ and a continuous map $\epsilon:\bar{\mathbb D}\to \{\pm 1\}$ such that 
\[
T(f)=T(0)+\alpha[f\circ\pi]^{\epsilon},\quad f\in \Sd,
\]
where $\alpha=T(\1)-T(0)$ is a unimodular constant since $T(\1)-T(0)$ is unimodular function and it is analytic on ${\mathbb D}$. Furthermore $\epsilon=1$ on $\bar{\mathbb D}$ or $\epsilon=-1$ on $\bar{\mathbb D}$. Put $T_1=\bar\alpha(T-T(0))$. Then 
\[
T_1=f\circ \pi,\quad f\in \Sd
\]
if $\epsilon=1$, and 
\[
T_1(f)=\overline{f\circ \pi}, \quad f\in \Sd
\]
if $\epsilon=-1$. Letting $f=\id$, the identity function, we see that 
$\pi\in \Sd$ if $\epsilon=1$ and $\bar\pi\in \Sd$  if $\epsilon=-1$. Put $\varphi =\pi$ if $\epsilon=1$, and $\varphi=\bar\pi$ if $\epsilon=-1$. Then we have that $\varphi\in \Sd$ and 
\[
T_1(f)=f\circ \varphi, \quad f\in \Sd
\]
if $\epsilon =1$, and 
\[
T_1(f)=\overline{f\circ \bar{\varphi}},\quad f\in \Sd
\]
if $\epsilon= -1$. 
In particular, we have 
\begin{equation}\label{t1}
T_1(\id)=\varphi
\end{equation}
either for $\epsilon=1$ and for $\epsilon=-1$. 
Since $T_1$ is pointwise 2-local in $\Iso(\Sd,\Sd)$ by the definition of $T_1$, for every $x\in \bar{\mathbb D}$ there exists $u_x\in \Sd$ and unimodular constant $\alpha_x, \lambda_x$ such that 
\[
(T_1(\id))(x)=u_x(x)+\alpha_x\id(\lambda_xx)
\]
and 
\[
0=(T_1(0))(x)=u_x(x),
\]
or
\[
(T_1(\id))(x)=u_x(x)+\alpha_x\overline{\id(\overline{\lambda_xx})}=u_x(x)+\alpha_x\id(\lambda_xx)
\]
and 
\[
0=(T_1(0))(x)=u_x(x).
\]
In any case we have
\begin{equation}\label{t11}
(T_1(\id))(x)=\alpha_x\lambda_xx.
\end{equation}
Combining \eqref{t1} and \eqref{t11} we have
\[
\varphi(x)=\alpha_x\lambda_xx
\]
 for every $x\in \bar{\mathbb D}$. 
Then we have $\varphi(0)=0$, and $|\varphi(x)|=|x|$ for every $x\in \bar{\mathbb D}$.
Since $\varphi:\bar{\mathbb D}\to \bar{\mathbb D}$ is analytic in ${\mathbb D}$,    the Schwartz lemma asserts that there is a unimodular constant $\lambda_0$ such that\[
\varphi(x)=\lambda_0x,\quad x\in \bar{\mathbb D}.
\]
It follows that 
\[
T(f)=T(0)+(T(\1)-T(0))f(\lambda_0\cdot),\quad f\in \Sd
\]
or
\[
T(f)=T(0)+(T(\1)-T(0))\overline{f(\overline{\lambda_0\cdot})},\quad f\in \Sd.
\]
By Theorem \ref{miura} we conclude that $T\in \Iso(\Sd,\Sd)$.
\end{proof}
\section{Iso-reflexivity}
Many literatures study isometries from the point of view
of how they are determined by their local actions \cite{bamol,cs,jcv,molarch,molgyo,molzal,so}.
By Theorem \ref{2} we have that several 2-local maps are linear, hence they are local maps.
In this section we prove that a local isometry in $\isonlyc(B_1,B_2)$ is 2-local in $\isonly(B_1,B_2)$. Applying Theorem we see the reflexivity of $\isonlyc(B_1,B_2)$ for several Banach spaces of continuous functions.
\begin{definition}
Put
\[
\text{
$\Mmc=\{T\in M(B_1,B_2); T$ is complex-linear $\}$}
\]
\[
\text{
$\isonly_{\mathbb C}(B_1,B_2)=\{T\in \isonly(B_1,B_2); T$ is complex-linear $\}$}.
\]
Recall that 
 $T\in \Mmc$ is local in $\isonlyc(B_1,B_2)$ if for every $f\in B_1$, there exists $T_f\in \isonlyc(B_1,B_2)$ such that 
\[
T(f)=T_f(f).
\]
\end{definition}

We say that $\isonlyc(B_1,B_2)$ is iso-reflexive if every local map in $\isonlyc(B_1,B_2)$ is an element in $\isonlyc(B_1,B_2)$.
\begin{prop}\label{localis2-local}
Suppose that $T\in \Mmc$ is local in $\isonlyc(B_1,B_2)$. Then $T$ is 2-local in $\isonly(B_1,B_2)$.
\end{prop}
\begin{proof}
Let $f,g\in B_1$ be arbitrary. Then there exists $T_{f,g}\in \isonlyc(B_1,B_2)$ such that 
\[
T(f-g)=T_{f,g}(f-g).
\]
As $T$ and $T_{f,g}$ are complex-linear, we have
\begin{equation}\label{grgr}
T(f)-T(g)=T_{f,g}(f)-T_{f,g}(g).
\end{equation}
Put 
\[
h_{f,g}=T(f)-T_{f,g}(f).
\]
By \eqref{grgr} we have 
\[
T(f)=h_{f,g}+T_{f,g}(f),
\]
\[
T(g)=h_{f,g}+T_{f,g}(g).
\]
It is easy to see that $h_{f,g}+T_{f,g}(\cdot) \in\isonly(B_1,B_2)$. It follows that $T$ is 2-local in $\isonly(B_1,B_2)$.
\end{proof}
\begin{theorem}
Suppose that every 2-local map in $\isonly(B_1,B_2)$ is an element in $\isonly(B_1,B_2)$. Then $\isonlyc(B_1,B_2)$ is iso-reflexive.
\end{theorem}
\begin{proof}
Suppose that $T\in \Mmc$ is local in $\isonlyc(B_1,B_2)$. Then by Proposition \ref{localis2-local}, $T$ is 2-local in $\isonly(B_1,B_2)$. By assumption, we have $T\in \isonly(B_1,B_2)$. Since $T$ is complex-linear, we infer that $T \in \isonlyc(B_1,B_2)$.
\end{proof}
Applying Corollaries \ref{C(X)},\ref{diskalgebra},\ref{Lip2},\ref{c101} and \ref{S}, we see that 
$\isonlyc(C(X_1),C(X_2))$ for first countable compact Hausdorff spaces $X_1$ and $X_2$, $\isonlyc(A(\bar{\mathbb D}),A(\bar{\mathbb D}))$, $\isonlyc(\Lip(X_1),\Lip(X_2))$,  
$\isonlyc(C^1[0,1],C^1[0,1])$ and $\isonlyc(\Sd,\Sd)$ are iso-reflexive.



\begin{thebibliography}{99}
\bibitem{ahhf}
H.~Al-Halees and R.~Fleming,
\emph{
On 2-local isometries on continuous vector valued function spaces},
J. Math. Anal. Appl. {\bf 354} (2009), 70--77
doi:10.1016/j.jmaa.2008.12.023

\bibitem{bamol}
C.~J.~K.~Batty and L.~Moln\'ar,
\emph{
On topological reflexivity of the groups of $*$-automorphisms and surjective isometries of $B(H)$},
Arch. Math. {\bf 67} (1996), 415--421

\bibitem{bjm}
F.~Botelho, J.~Jamison and L.~Moln\'ar,
\emph{
Algebraic reflexivity of isometry groups and automorphism groups of some operator structures}
J. Math. Anal. Appl. {\bf 408} (2013), 177--195
doi:10.1016/j.jmaa.2013.06.001

\bibitem{cp}
J.~C.~Cabello and A.~M.~Peralta,
\emph{
Weak-2-local symmetric maps on $C^*$-algebras},
Linear Algebra Appl. {\bf 494} (2016), 32--43
doi:10.1016/j.laa.2015.12.024


\bibitem{cs}
F.~Cabello S\'anchez,
\emph{
Automorphisms of algebras of smooth functions and equivalent functions},
Differential Geom. Appl. {\bf 30} (2012), 216--221
doi:10.1016/j.difgeo.2012.03.003

\bibitem{gleason}
A.~M.~Gleason,
\emph{
A characterization of maximal ideals}, 
J. Analyse Math. {\bf 19} (1967), 171--172

\bibitem{gy}
M.~Gy\H{o}ry,
\emph{2-local isometries of $C_0(X)$},
Acta Sci. Math. (Szeged) {\bf 67} (2001), 735--746

\bibitem{hm}
O.~Hatori and T.~Miura,
\emph{
Real linear isometries between function algebras. II},
Cent. Eur. J. Math. {\bf 11} (2013), 838--1842
doi:10.2478/s11533-013-0282-0

\bibitem{hmto}
O.~Hatori, T.~Miura, H.~Oka and H.~Takagi,
\emph{
2-Local Isometries and 2-Local Automorphisms
on Uniform Algebras},
Int. Math. Forum {\bf 50} (2007), 2491--2502
doi:10.12988/imf.2007.07219

\bibitem{haoi}
O.~Hatori and S.~Oi,
\emph{
2-local isometries on function spaces},
to appear in Contemp. arXiv:1812.10342

\bibitem{jcv}
A.~Jim\'enez-Vargas, A.~Morales Compoy and M.~Villegas-Vallecillos,
\emph{
Algebraic reflexivity of the isometry group of some spaces of Lipschitz functions},
J. Math. Anal. Appl. {\bf 366} (2010), 195--201
doi:10.1016/j.jmaa.2010.01.034

\bibitem{jlp}
A.~Jim\'enez-Vargas, L.~Li, A.~M.~Peralta, L.~Wang and Y.-S~Wang,
\emph{
2-local standard isometries on vector-valued Lipschitz function spaces},
 J. Math. Anal. Appl. {\bf 461} (2018), 1287--1298
doi:10.1016/j.jmaa.2018.01.029

\bibitem{jvvv}
A.~Jimenez-Vargas and M.~Villegas-Vallecillos,
\emph{
2-local isometries on spaces of Lipschitz functions},
Canad. Math. Bull. {\bf 54} (2011), 680--692
doi:10.4153/CMB-2011-25-5

\bibitem{kz1}
J.~P.~Kahane and W.~\.Zelazko
\emph{
A characterization of maximal ideals in commutative Banach algebras},
Studia Math. {\bf 29} (1968), 339--343

\bibitem{kkm}
K.~Kawamura, H.~Koshimizu and T.~Miura,
\emph{
Norms on $C^1([0,1])$ and their isometries},
Acta Sci. Math. (Szeged), {\bf 84} (2018), 239--261 
doi:10.14232/actasm-017-331-0


\bibitem{ks}
S.~Kowalski and Z. S\l odkowski,
\emph{
A characterization of multiplicative linear functionals in Banach algebras},
Studia Math. {\bf 67} (1980), 215--223


\bibitem{lpww}
L.~Li. A~.M.~Peralta, L.~Wang and Y.-S~Wang,
\emph{
Weak-2-local isometries on uniform algebras and Lipschitz algebras}
Publ. Mat. {\bf 63} (2019), 241--264
doi:10.5565/PUBLMAT6311908

\bibitem{ma}
P.~Mankiewicz,
\emph{
On the differentiability of Lipschitz mappings in Fr\'echet spaces},
Studia Math.{\bf 45}(1973), pp.15--29




\bibitem{mi}
T.~Miura,
\emph{
Surjective isometries on a Banach space of analytic functions on the open unit disc},
preprint, 2019,arXiv:1901.02737 

\bibitem{mt}
T.~Miura and H.~Takagi,
\emph{
Surjective isometries on the Banach space of continuously differentiable functions},
Contemp. Math. 
{\bf 687} (2017), 181--192 
doi:10.1090/conm/687/13787


\bibitem{mol2loc2}
L.~Moln\'ar,
\emph{
On 2-local *-automorphisms and 2-local isometries of ${\it B}(H)$}, to appear in J. Math. Anal. Appl.

\bibitem{mo2loc}
L.~Moln\'ar,
Private communication with O. Hatori, 2018

\bibitem{molbook}
L.~Moln\'ar,
\emph{
Selected Preserver Problems on Algebraic Structures of Linear operators and on Function Spaces},
Springer, Berlin, 2007

\bibitem{mol2l}
L.~Moln\'ar,
\emph{
2-local isometries of some operator algebras
},
Proc. Edinb. Math. Soc. (2) {\bf 45} (2002), 349--352 
doi:10.1017/S0013091500000043

\bibitem{molarch}
L.~Moln\'ar,
\emph{
Reflexivity of the automorphism and isometry groups of $C^*$-algebra in BDF theory},
Arch. Math. {\bf 74} (2000), 120--128

\bibitem{molgyo}
L.~Moln\'ar and M.~Gy\"ory,
\emph{
Reflexivity of the automorphism and isometry groups of the suspension of $B(H)$},
J. Funct. Anal. {\bf 159} (1998), 568--586
doi:10.1006/jfan.1998.3325

\bibitem{molzal}
L.~Moln\'ar and B.~Zalar,
\emph{
Reflexivity of the group of surjective isometries on some Banach spaces},
Proc. Edinb. Math. Soc. {\bf 42} (1999), 17--36
doi:10.1017/S0013091500019982



\bibitem{hamo}
M.~Mori,
\emph{
On 2-local nonlinear surjective isometries on normed spaces and C*-algebras},
preprint, arXiv:1907.02172

\bibitem{so}
S.~Oi,
\emph{
Algebraic reflexivity of isometry groups of algebras of Lipschitz maps
},
Linear Algebra Appl. {\bf 566} (2019), 167--182
doi:10.1016/j.laa.2018.12.033


\bibitem{smrl}
P.~\v Semrl,
\emph{
Local automorphisms and derivations on $B(H)$},
Proc. Amer. Math. Soc. {\bf 125} (1997), 2677--2680
doi:10.1090/S0002-9939-97-04073-2

\bibitem{zelazko}
W.~\.Zelazko,
\emph{
A characterization of multiplicative linear functionals in complex Banach algebras}, 
Studia Math. {\bf 30} (1968), 83--85


\end{thebibliography}
\end{document}